\documentclass[reqno]{amsart}
\newcommand \datum {October 29, 2023}

\usepackage{amssymb,latexsym}
\usepackage{amsmath}
\usepackage{hyperref}
\usepackage{url}
\usepackage{graphicx}
\usepackage[dvipsnames]{xcolor}
\usepackage{enumerate}
\numberwithin{equation}{section}
\theoremstyle{plain}
 \newtheorem{theorem}{Theorem}[section]

 \newtheorem{proposition}[theorem]{Proposition}

 \newtheorem{corollary}[theorem]{Corollary}
\theoremstyle{definition}

 \newtheorem{remark}[theorem]{Remark}
\theoremstyle{remark}

\newcommand \ideal [1]{\mathord\downarrow #1}
\newcommand \filter[1]{\mathord\uparrow #1}

\newcommand \clnp {\textbf{NP}}
\newcommand \clp {\textbf{P}}
\newcommand \Alg   {\mathcal A}
\newcommand \Blg   {\mathcal B}
\newcommand \Mlg   {\mathcal M}
\newcommand \Enc {E}

\newcommand \CPr {\textup{CPr}}
\newcommand \DPr {\textup{DPr}}
\newcommand \defiff{\mathrel{\overset{\textup{def}}\iff}}
\newcommand\sg {s_G}

\newcommand\size[1]{\textup{size}(#1)}
\newcommand\Quo[1]{\textup{Quo}(#1)}
\newcommand\Equ[1]{\textup{Equ}(#1)}
\newcommand\Boo[1]{\mathsf B_{#1}}
\newcommand\At[1]{\textup{At}(#1)}
\newcommand\Spsign{\textup{Sp}}
\newcommand\Sp[1]{\Spsign(#1)}
\newcommand\LASp[1]{\textup{LASp}(#1)}

\newcommand\lint[1]{\lfloor #1\rfloor}
\newcommand\FS [1] {\textup{FS}_\wedge(#1)}

\newcommand \svec [1] {\vec{#1}\kern1.5pt}
\newcommand \pvec [1] {\vec{#1}\kern1.5pt'}

\newcommand \FD [1] {\textup{FD}(#1)}
\renewcommand \phi{\varphi}

\newcommand \Nplu {\mathbb N^+}
\newcommand \then {\Rightarrow}

\newcommand{\tbf}{\textbf}
\newcommand{\set}[1]{\{#1\}}

\newcommand \nothing[1] {}
\newcommand \red[1]{{\textcolor{red}{#1}\color{black}}}

\begin{document}

\title[Generating Boolean lattices and key exchange]
{Generating Boolean lattices by few elements and exchanging session keys}

\author[G.\ Cz\'edli]{G\'abor Cz\'edli}
\email{\href{czedli@math.u-szeged.hu}{czedli@math.u-szeged.hu}}
\urladdr{\href{http://www.math.u-szeged.hu/~czedli/}{http://www.math.u-szeged.hu/{\textasciitilde}czedli/}}
\address{University of Szeged, Bolyai Institute. 
Szeged, Aradi v\'ertan\'uk tere 1, HUNGARY 6720}

\begin{abstract} 
Let $\Sp k$ denote the number of the $\lint{k/2}$-element subsets of a finite $k$-element set. We prove that the least size of a generating subset of the Boolean lattice with $n$ atoms (or, equivalently, the powerset lattice of an $n$-element set) is the least number $k$ such that $n\leq \Sp k$. 
Based on this fact and our 2021 protocol based on equivalence lattices, we  outline a cryptographic protocol for exchanging session keys, that is, frequently changing secondary keys. In the present paper, which belongs mainly to lattice theory, we do not elaborate and prove those details of this protocol that modern cryptology would require to guarantee security; the security of the protocol relies on heuristic considerations. However, as a first step, we prove that if an eavesdropper could break every instance of an easier protocol in polynomial time,  then  \clp{} would equal \clnp{}. As a byproduct, it turns out that in each nontrivial finite lattice that has a prime filter, in particular, in each nontrivial finite Boolean lattice, the solvability of systems of equations with constant-free left sides but constant right sides is an \clnp-complete problem. 
\end{abstract}

\thanks{This research was supported by the National Research, Development and Innovation Fund of Hungary, under funding scheme K 138892.  \hfill{\red{\tbf{\datum}}}}






\maketitle

\section{Introduction}

\subsection{Targeted readership}
This is mainly a \emph{lattice theoretic paper} and the a main result belongs to lattice theory. However, the targeted readership is not restricted to lattice theorists. Those familiar with the concept of a Boolean lattice and that of \clnp{}-completeness should have no difficulty in reading the \emph{results} and even most other parts of the paper.
Some exceptions occur in Subsection \ref{subsect:msrvy}, which surveys how a series of lattice theoretic  investigations lead to the present paper, which could be interesting outside lattice theory and, perhaps, even outside mathematics.

\subsection{Our goal} As usual, $\Nplu=\set{1,2,\dots}$ stands for the set of positive integers. For $n\in\Nplu$, 
let $\Boo n=(\Boo n;\vee,\wedge)$ be the \emph{Boolean lattice} with $n$ atoms. Note that $\Boo n$ is isomorphic to (and so it can be defined as) the powerset lattice $(P(\set{1,\dots,n});\cup,\cap)$, whence $|\Boo n|=2^n$. 
A subset $X$ of $\Boo n$ is a \emph{generating set} of $\Boo n$ if no proper subset of $\Boo n$ that is  closed with respect to join ($\vee$) and meet ($\wedge$) includes $X$ as a subset. In Theorem \ref{thm:nleqSpk}, we are going to determine the smallest $k\in\Nplu$ such that $\Boo n$ has a $k$-element generating set.
Section \ref{sect:prrslt}, which  is a computer-assisted, indicates that if $k'>k$ but $k'$ is still small, then $\Boo n$ has  many $k'$-element generating sets.
Based on the plenty of these generating sets, Section \ref{sect:ACprtcl} outlines a 
cryptographic protocol for session key exchange. Section \ref{sect:NPh} shows that  if an eavesdropper (that is, a monitoring adversary) could solve every instance of a related but easier lattice theoretic problem in polynomial time,  then  \clp{} would equal \clnp{}.  Finally, Section \ref{section:warn} warns the reader that this connection with \clnp{} does not guarantee security in itself and, on the positive side, Section \ref{section:warn} shows some perspectives.

\subsection{A historical mini-survey}\label{subsect:msrvy}
From the author's perspective, the story started with Z\'adori \cite{zadori}, who gave a new proof of a result of Strietz \cite{strietz75}--\cite{strietz77} asserting that the \emph{equivalence lattice} $\Equ A$ (consisting of all equivalences of $A$) has a 4-element generating set provided that $A$ is
a finite set and $|A|\geq 3$. For short, we say that $\Equ A$ is \emph{$4$-generated} for these finite sets $A$. 
In the next step, based on Z\'adori's method, Chajda and Cz\'edli \cite{chajdaczg} proved that the lattices $\Quo A$ of all quasiorders (AKA preorders) of these finite sets $A$ and even some infinite sets $A$ are 6-generated; in fact, they are 3-generated if we add the unary operation $\rho\mapsto \rho^{-1}=\set{(y,x): (x,y)\in \rho}$
of forming inverses to the set $\set{\bigvee,\bigwedge}$ of infinitary lattice operations. Next, Cz\'edli \cite{CzGEateq} extended Z\'adori's result to  $\Equ A$ with $|A|=\aleph_0$. Furthermore, Cz\'edli \cite{CzGEq4,czgeek} and Tak\'ach \cite{takach} proved that $\Equ A$ and $\Quo A$ are 
4-generated and 6-generated, respectively, provided that $A$ is an infinite set  and there is no inaccessible cardinal $\lambda$ such that $\lambda\leq |A|$. Moreover, the 1999 paper Cz\'edli \cite{czgeek} proved that $\Equ A$ has a 4-element non-antichain generating set for these sets $A$. 
Note that Kuratowski \cite{kuratowski} gave a model of  ZFC in which there is no inaccessible cardinal at all.

Around 1999, Vilmos Totik\footnote{\href{https://en.wikipedia.org/wiki/Vilmos_Totik}{https://en.wikipedia.org/wiki/Vilmos\textunderscore{}Totik}} proved that our methods are insufficient to deal with inaccessible cardinals. Hence, the topic was put aside after the 1999 paper Cz\'edli \cite{czgeek}, and  it is still an open problem whether $\Equ A$ and $\Quo A$ are finitely generated (as complete lattices) if there exists an inaccessible cardinal $\leq |A|$. 

The research started again in 2015, when Dolgos \cite{dolgos}, one of Mikl\'os Mar\'oti's students, proved that $\Quo A$ is 5-generated for $|A|\leq \aleph_0$, and  Kulin \cite{kulin} extended this result to all sets $A$ such that there is no inaccessible cardinal $\lambda\leq |A|$. Not much later, Cz\'edli \cite{czg2017fourgen} and Cz\'edli and Kulin \cite{czgkulin} reduced the number of generators by proving that for all sets $A$ such that 
$|A|\neq 4$ and there is no inaccessible cardinal $\lambda\leq |A|$,  the complete lattice $\Quo A$ is 4-generated. The case $|A|=4$ is still open but the result was optimal for many other sets, as \cite{czg2017fourgen} proved that $\Quo A$ is not 3-generated if $|A|\geq 3$. 
Finding 4-element generating sets that are not antichains is more difficult but, after Strietz \cite{strietz75}--\cite{strietz77} and Z\'adori \cite{zadori}, some sporadic cases have recently been settled in  Ahmed and Cz\'edli \cite{delbrinczg} and Cz\'edli and Oluoch \cite{czgoluoch}.

In 2020, it appeared that the technique developed for infinite sets is appropriate to show that even some direct powers and products of some finite equivalence lattices are 4-generated and (consequently) $\Equ A$ and $\Quo A$ have very many 4-element generating sets if $|A|$ is a large finite number; see Cz\'edli \cite{czgDEBRauth} and Cz\'edli and Oluoch  \cite{czgoluoch}. Based on the abundance of the generating sets found in the just mentioned two papers,  Cz\'edli \cite{czgDEBRauth} in 2021 suggested a protocol (the \emph{2021 protocol} for short) for authentication and cryptography based on lattices. The questions of security of this protocol gets easier when the protocol is used only for session key exchange, because then we can practically disregard those adversaries that alter messages;  recall from Buegler \cite{buegler} that a \emph{session key} is a secondary key to be changed before each usage of a cryptographic protocol while the \emph{master key} remains unchanged. 
In other words, while Cz\'edli \cite{czgDEBRauth} puts the emphasis on authentication, now we have a reason to put it on session key exchange. 
Quite recently, while looking for small generating sets of some filters of quasiorder lattices, a proof in Cz\'edli \cite{czgfiltQuo} required to know the smallest size of a generating set of a finite Boolean lattice; this was the immediate motivation for  the present paper. 

Out of the several directions where the present paper will possibly lead, we only mention the following. A \emph{weak congruence lattice} of an algebra $\mathcal A=(A;F)$ is a bunch  of congruence lattices along a scaffolding, which is the subalgebra lattice of $\mathcal A$; see, e.g.,  \v{S}e\v{s}elja, Stepanovi\'c, and Tepav\v{c}evi\'c \cite{seseljaatal} for an exact definition and references. Now if $A$ is finite and  $F=\emptyset$, then we know from Z\'adori \cite{zadori} that the congruence lattices in question are equivalence lattices generated  by few (at most four) elements, and the  scaffolding is also generated by few elements by the main result of the present paper. This raises the question how many of its elements are needed to generate the weak congruence lattice in the $F=\emptyset$ case.

\section{Small generating sets of finite Boolean  lattices}\label{sec:genBoole}
For $n\in\Nplu$, we introduce the ``vertical-space-saving'' notation
\begin{equation}
\Sp n:={{n}\choose{\lint{n/2}}}
\end{equation}
where $\lint{n/2}$ is the (lower) integer part of $n/2$. For example, 
\begin{equation}
\Sp{32}= 601\,080\,390 \text{ \ and \ }\Sp{33}=  1\,166\,803\,110.
\label{eq:nLmgpGNkw}
\end{equation}
The notation $\Spsign${} comes from ``Sperner''; see later. For $n\in\Nplu$, let $\LASp n$ be the smallest $k\in\Nplu$ such that $n\leq \Sp k$. Note the rule: $n\leq \Sp k \iff \LASp n\leq k$; this explains the acronym, which comes from ``Left Adjoint of Sp".

\begin{theorem}\label{thm:nleqSpk}
For $n,k\in\Nplu$, $\Boo n$ has an at most $k$-element generating set if and only if $n\leq \Sp k$ or, equivalently,  if and only if $\,\LASp n\leq k$. In particular, $\LASp n$ is the smallest possible size of a generating set of $\Boo n$.
\end{theorem}

For example, this theorem together with \eqref{eq:nLmgpGNkw} give that  $\Boo{1\,000\,000\,000}$ is $33$-generated but not $32$-generated.

\begin{proof} Let $\At{\Boo n}$ be the set of atoms of $\Boo n$. As usual, for  an element $u$ of a lattice $L$,  $\ideal u$ and $\filter u$ will stand for  $\set{x\in L: x\leq u}$ and $\set{x\in L: x\geq u}$, respectively. First, we show that for any subset $Y$ of $\Boo n$,
\begin{equation}
\text{if }Y\text{ generates }\Boo n\text{ and }  a\in\At{\Boo n}\text{, then }a=\bigwedge (Y\cap\filter a). 
\label{eq:ktNrMxtlgRn}
\end{equation}
As $Y$, say $Y=\set{b_1,\dots,b_m}$,  generates $\Boo n$ and $\Boo n$ is distributive,  $a=t(b_1,\dots,b_m)$ for  an $m$-ary disjunctive normal form, that is, $a$ is the join of meets of elements of $Y$. But $a$ is join-irreducible, whereby it is the meet of some elements of $Y$. This shows the ``$\geq$'' part of \eqref{eq:ktNrMxtlgRn}. The ``$\leq$'' is trivial, and we have proved  \eqref{eq:ktNrMxtlgRn}. 

Next, we  claim that for any subset $G$ of $\Boo n$,
\begin{equation}
\text{if $G$ generates $\Boo n$ and $k=|G|$,  then $n\leq \Sp k$.}
\label{eq:BrmpCtznnmKf}
\end{equation}
To show this, assume that $G$ is a $k$-element generating set of $\Boo n$. Let $X$ be a $k$-element set and denote by $\FS X$ the meet-semilattice freely generated by $X$. Denote by $M$ the meet-subsemilattice of $(\Boo n;\wedge)$ generated by $G$. Pick a bijective map $f_0\colon X\to G$. 
The freeness of $\FS X$ allows us to extend $f_0$ to a meet-homomorphism $f\colon\FS X\to M$, which is surjective since $f(X)=G$ generates $M$.  By \eqref{eq:ktNrMxtlgRn}, $\At{\Boo n}\subseteq M$. This together with the surjectivity of $f$ allow us to take an injective map $g\colon\At{\Boo n}\to \FS X$ such that, for all $a\in \At{\Boo n}$,  $f(g(a))=a$.  If we had that $g(a)\leq g(a')$ for distinct $a,a'\in \At{\Boo n}$, then $g(a)=g(a)\wedge g(a')$ would lead to $a=f(g(a))=f(g(a)\wedge g(a'))= f(g(a))\wedge f(g(a'))=a\wedge a'$, yielding that $a\leq a'$ and contradicting that $a$ and $a'$ are distinct atoms of $\Boo n$. 
Therefore $g(a)\parallel g(a')$, that is, $g(\At{\Boo n})$ is an $n$-element antichain in $\FS X$. Adding a top element to $\FS X$, we obtain another semilattice,  $\set{1}\cup\FS X$. We know from the folklore or from McKenzie, McNulty, and Taylor \cite[Page 240, \S4]{mcKmcNT} that $\set{1}\cup\FS X$  is order isomorphic to $\Boo{|X|}=\Boo k$. So $\Boo k$ has an $n$-element antichain. By Sperner's theorem \cite{sperner}, see also 
Gr\"atzer \cite[page 354]{GLTFound},  any antichain in $\Boo k$ has at most $\Sp k$ elements. This implies  \eqref{eq:BrmpCtznnmKf} and the ``only if'' part of the theorem.

Next, observe that 
\begin{equation}\left.
\parbox{10.6cm}{for any $m\leq n\in\Nplu$,  $\Boo m$ is a homomorphic image of $\Boo n$. Therefore, if $\Boo n$ has an at most $k$-element generating set, then so does $\Boo m$.}\,\,\right\}
\label{eq:wPsrTwgftjG}
\end{equation}
It suffices to show  the first part  for $m=n-1$. Let $c$ be a coatom (that is, a lower cover of 1) in $\Boo n$. Then $\ideal c\cong \Boo m$. The function $f\colon \Boo n\to \ideal c$  defined by $x\mapsto c\wedge x$ is a homomorphism by distributivity. As $x=f(x)$ for each $x\in \ideal c$,  we conclude \eqref{eq:wPsrTwgftjG}.

Next, to show the ``if'' part of the theorem, assume that $n\leq \Sp k$; we are going to show that $\Boo n$ has an at most $k$-element generating set. 
Based on \eqref{eq:wPsrTwgftjG}, we  can assume that $n=\Sp k$.
As $\Boo k$ is isomorphic to the powerset lattice $(P(\set{1,\dots,k});\cup,\cap)$ and the $\lint{k/2}$-element subsets of $\set{1,\dots,k}$ form an $n=\Sp k$-element antichain in $(P(\set{1,\dots,k});\cup,\cap)$, it follows that $\Boo k$ has an $n$-element antichain $H$. As $(P(H);\cup,\cap)\cong \Boo n$, 
it suffices to find a $k$-element generating set of the powerset lattice $P(H)=(P(H);\cup,\cap)$. 
For each $a\in\At{\Boo k}$, we let $X_a:=H\cap\filter a$. Then $X_a\in P(H)$ and $G:=\set{X_a: a\in \At{\Boo k}}$ is an at most $k$-element subset of $P(H)$.  To show that $G$ generates  $P(H)$,  it suffices to show that for every $h\in H$, 
\begin{equation}
\set{h}=\bigcap\set{X_a: a\in \At{\Boo k}\cap \ideal h}.
\label{eq:rZmszltpLrnts}
\end{equation}
For every $a\in \At{\Boo k}\cap \ideal h$, we have that $h\in H\cap\filter a=X_a$, showing the ``$\subseteq$'' part of \eqref{eq:rZmszltpLrnts}. 
Now assume that $h'\in H$ belongs to the intersection in \eqref{eq:rZmszltpLrnts}. Then $h'\in X_a$ for every  $a\in \At{\Boo k}$ such that $a\leq h$.
Writing this in a more useful way, 
\begin{equation*}
(\forall a\in \At{\Boo k})\,\,\bigl( a\leq h \then a\leq h'\bigr)\text{, that is, } \At{\Boo k}\cap\ideal h\subseteq \At{\Boo k}\cap\ideal {h'}.
\end{equation*}
Hence, using that each element of $\Boo k$ is the join of all atoms below it,  $h=\bigvee(\At{\Boo k}\cap\ideal h)\leq \bigvee(\At{\Boo k}\cap\ideal{h'})=h'$. 
But $h,h'\in H$ and $H$ is an antichain,  whereby $h\leq h'$ gives that $h'=h\in\set{h}$, showing the ``$\supseteq$'' part of \eqref{eq:rZmszltpLrnts}. Therefore, \eqref{eq:rZmszltpLrnts} and the ``if'' part of the theorem hold.
\end{proof}

\begin{corollary} If $2\leq k\in \Nplu$ and $n \leq \Sp k$, then the free distributive lattice $\FD k$ has a sublattice isomorphic to $\Boo n$.
\end{corollary}

\begin{proof} As $\Boo m$ is a sublattice of $\Boo n$ for any $m\leq n$, we can assume that $n=\Sp k$. Theorem \ref{thm:nleqSpk} yields a surjective homomorphism $f\colon\FD k\to \Boo n$. Let $h\colon \Boo n\to\Boo n$ be the identity map (defined by $x\mapsto x$ for $x\in\Boo n$). Since $\Boo n$ is projective in the class of all distributive lattices by Balbes \cite[Theorem 7.1(i),(iii')]{balbes}, there is a homomorphism $g\colon \Boo n\to \FD k$ such that  $f g=h$. As the product $h$ is injective, so is $g$. Thus, $g(\Boo n)\cong \Boo n$ and  $g(\Boo n)$ is a required sublattice of $\FD k$. 
\end{proof}

\section{The abundance of small generating sets of finite Boolean lattices}\label{sect:prrslt}
We call a $k$-dimensional vector $\vec h=(h_1,\dots, h_k)$ a \emph{generating vector} of $\Boo n$ if the set $\set{h_1,\dots,h_k}$ of its components is a generating set of $\Boo n$. Here $|\set{h_1,\dots,h_k}|\leq k$ and no equality is required. If $k<n$, then $k$ is much smaller than $|\Boo n|=2^n$, whereby the components of a randomly chosen $k$-dimensional vector from $\Boo n^k$ are pairwise distinct with high probability. 
Therefore, the ratio of the $k$-element generating sets to all $k$-element subsets of $\Boo n$ is close to the ratio of the $k$-dimensional generating vectors to all $k$-dimensional vectors belonging to $\Boo n^k$. 

A computer program, written by the author and available from his \href{http://www.math.u-szeged.hu/~czedli}{website} or from the appendix section of \href{https://arxiv.org/abs/2303.10790v2}{arXiv:2303.10790v2} (the July 26, 2023 version of this paper)  counted the generating vectors of $\Boo{1000}$ among one hundred thousand randomly selected  $k$-dimensional vectors for some values of $k$. Some of these computational results are given below and even more in  \href{https://arxiv.org/abs/2303.10790v2}{arXiv:2303.10790v2}. 

\noindent
{\footnotesize 
\begin{verbatim}
n=1000 k=40  Tested:100000 Generating:    42;      506.867 seconds.
n=1000 k=50  Tested:100000 Generating: 59003;     1305.780 seconds.
n=1000 k=80  Tested:100000 Generating: 99990;     2647.147 seconds.
n=1000 k=90  Tested:100000 Generating: 99999;     2974.364 seconds.
n=1000 k=100 Tested:100000 Generating:100000;     3265.869 seconds.
\end{verbatim}
}
\noindent
Thus, we conjecture that  a random member of $\Boo{1000}^{50}$ is a 50-dimensional generating vector of $\Boo{1000}$ with probability at least $1/2$. Note that  $\LASp{1000}=13$.


\section{A cryptographic protocol for session key exchange and encrypted communication}\label{sect:ACprtcl}
In this section, we outline how to tailor the 2021 protocol, see Cz\'edli \cite{czgDEBRauth}, from equivalence lattices to Boolean lattices but, as opposed to \cite{czgDEBRauth}, now we put the emphasis on session key exchange. The reader need not be an expert of cryptology. We  present only the main ideas in the paper; the caveats  in Section  \ref{section:warn} warn us that some details and proofs, which modern cryptology would expect, are still missing. As the session key exchange is (almost\footnote{\label{foot:clLtxr}It will be clear later that if a particular session key $\vec p(\vec h)$ is unused, then   $\vec p$, which the  Adversary can  intercept, gives  only the information mentioned in Footnote \ref{foot:ntlsGr}; this  is not enough to find $\vec h$ as there are astronomically many $k$-dimensional generating vectors.}) absolutely safe if the key remains unused, we are going to include the \emph{usage} of session keys in the suggested cryptographic protocols. This section and the protocols suggested in it rely only on heuristic considerations.  In order to mitigate this omen, note that the next section contains a proof of a related statement (and, hopefully, the heuristic considerations have some convincing value).  Furthermore, a session key  is to be used in some well-known cryptosystem like AES-256 (the 256-bit variant of Advanced Encryption Standard) or Vernam's cipher, and decades of experience show that these cryptosystems are safe when they are used appropriately. 
The full list of papers and other sources devoted to cryptosystems that can use session keys would possibly be longer than the present paper itself. Therefore, we reference only some introductory items\footnote{\label{foot:wikiitems}These Wikipedia items are the following: \href{https://en.wikipedia.org/wiki/Advanced_Encryption_Standard}{Advanced\textunderscore{}Encryption\textunderscore{}Standard}
, \href{https://en.wikipedia.org/wiki/Key-recovery_attack}{Key-recovery\textunderscore{}attack},
\href{https://en.wikipedia.org/wiki/Chosen_plaintext_attack}{Chosen\textunderscore{}plaintext\textunderscore{}attack},
\href{https://en.wikipedia.org/wiki/Session_key}{Session\textunderscore{}key}, 
\href{https://en.wikipedia.org/wiki/Gilbert_Vernam}{Gilbert\textunderscore{}Vernam} (they are clickable).}{} from Wikipedia (\href{https://www.wikipedia.org/}{https://www.wikipedia.org/}) that are sufficient to find lots of  related literature but note that the  reader need not read these items. However, as this concept occurs in the title, we mention that   the  terminology \emph{session key} comes from Buegler \cite{buegler}.

A \emph{session key} for a cryptosystem is a symmetric secret key used only once (or, at least, only very few times); symmetry means that encryption and decryption need the same key.  The importance of a session key is that while several traditional cryptosystems are vulnerable if the same key is used repeatedly, they are much safer if a key is used only once. We will always assume that
\begin{equation}
\parbox{11.0cm}{vectors of  element of $\Boo n$, session keys, and   \emph{plaintexts} (also called \emph{clear texts}) are  regarded \emph{bit vectors}, that is, vectors over the 2-element field when they are summands or a traditional cipher is applied to them.}
\label{eq:mrtmndksklD}
\end{equation}
With this convention (which is only a trivial question of encoding like ASCII), the well-known Vernam's cipher used with a session key
 $\vec c$ turns a clear text  $\vec x$ into the \emph{encrypted text} is $\vec x+\vec c$ (componentwise addition).  (We assume that  $\vec x$ does not have more bits than $\vec c$; if $\vec x$ has less bits than $\vec c$ then only the appropriate initial segment of $\vec c$ is used.)

In the subsequent two subsections, we give two versions of our protocol.

\subsection{The basic version}\label{subsect:strongest}
In our model, let $\Enc$ be a fixed traditional cryptosystem like Vernam's cipher or AES-256.
For a (secret) code $\vec u$ and a clear text $\vec x$, $\Enc(\vec u,\vec x)$ denotes the \emph{encrypted text} that $\Enc$ produces.
Kati\footnote{The  Hungarian variant of ``Cathy'' and ``Kate''.} communicates with her Bank online. 
They have previously agreed upon a \emph{secret master key} $\vec h$, which is a $k$-dimensional random generating vector $\vec h=(h_1,\dots, h_k)$ of $\Boo n$. This master key is a permanent key that both Kati and the Bank know and only they know; this is not a big restriction since most European banks do not allow anonymous clients, whereby Kati and the Bank have to meet before Kati opens a bank account. 
\begin{equation}
\parbox{9.5cm}{In addition to $\Enc$ and $\vec h$, there are three parameters in the model, $k$, $n$, and $b$. Let, say, $k:=50$, $n:=1000$, and $b:=100$.}
\label{eq:rmLgmBsm}
\end{equation}
When  Kati wants to send a clear text $\vec x=(x_1,\dots,x_b)\in \Boo n^b$ to the Bank, the first step is that she 
generates a random vector $\vec p=(p_1,\dots,p_b)$ of $k$-ary lattice terms or, rather,  she requests\footnote{\label{foot:bnkprvlgm} The Bank would certainly vindicate rights to generate $\vec p$ and would not allow Kati to do so.} such a vector  from the Bank; see \eqref{eq:rmLgmBsm}.
\begin{equation}
\begin{aligned}
\includegraphics[scale=0.23]{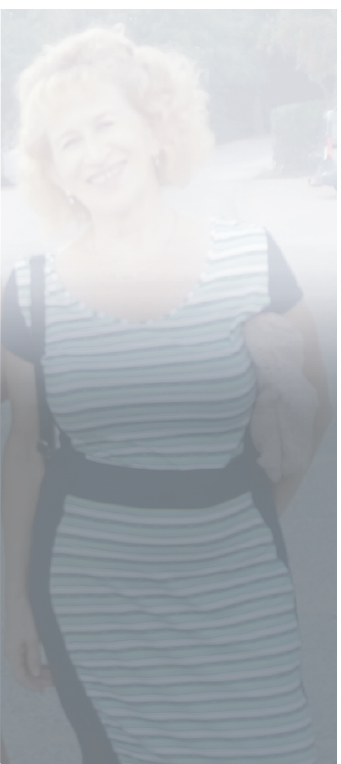}
\kern 13pt
\raisebox{21.5pt}
{\parbox{6.5cm}{
With this $\vec p$ obtained from the Bank, 
Kati computes  the \emph{session key} $\vec u:=\vec p(\vec h)=(p_1(\vec h),\dots, p_b(\vec h))$, uses the encryption $\Enc$ to obtain the encrypted text $\Enc(\vec u,\vec x)$, and sends $\Enc(\vec u,\vec x)$ together with $\vec p$ to the Bank.}}
\kern13pt
\includegraphics[scale=0.21]{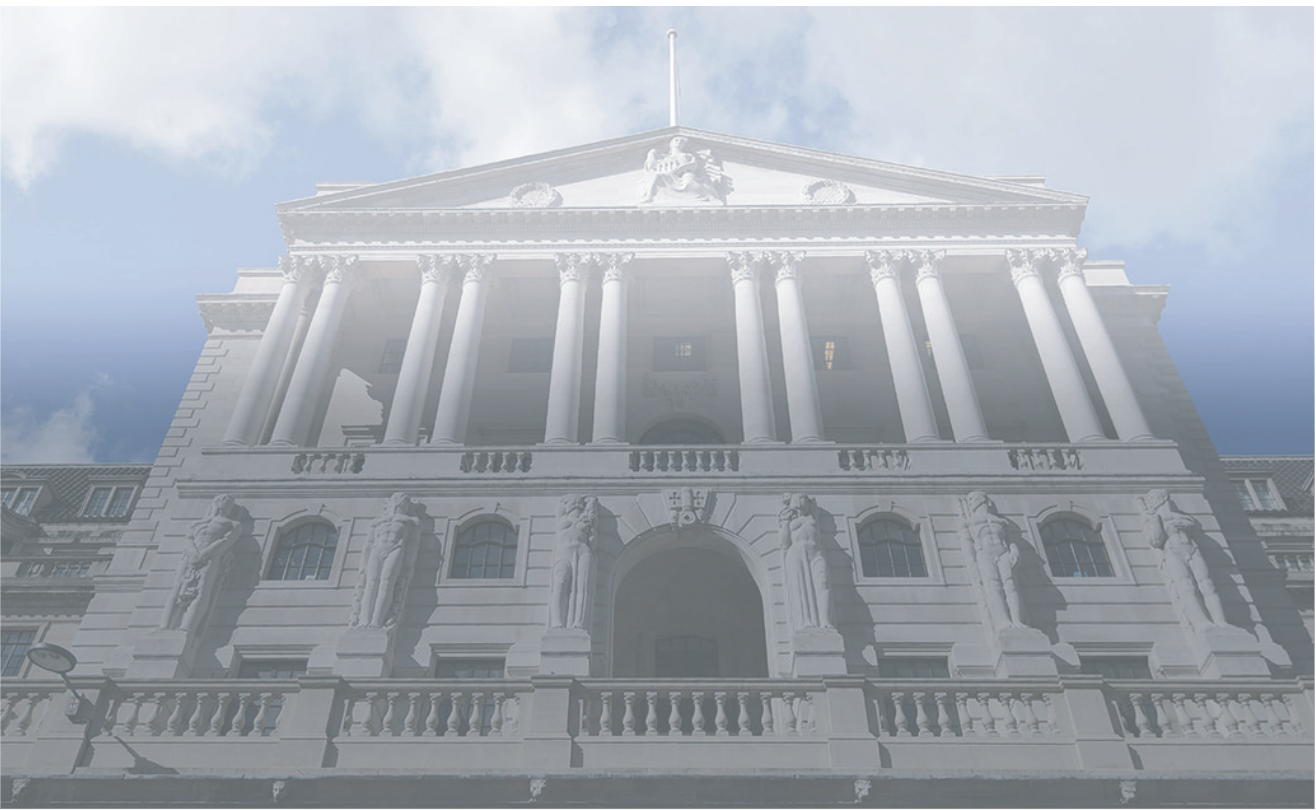}
\end{aligned}
\label{eq:lrJgfncNpZnkt}
\end{equation}
Changing their roles (except that the Bank generates $\vec p$), the Bank can also send an encrypted message to Kati. 
Too simple terms\footnote{\label{foot:ntlsGr}The concept of ``not too simple terms'' should be defined. For example, we can require that $\set{h_1,\dots,h_k} \cap \set{p_1(\vec h),\dots,p_b(\vec h)}=\emptyset$. Footnote 22 in \href{https://arxiv.org/abs/2303.10790v2}{arXiv:2303.10790v2} gives some vague ideas how to obtain random terms that are hopefully not too simple on the average.}  should be avoided, of course. There is an Adversary who  eavesdrops on the communication channel and intercepts messages. Furthermore, he can alter messages and pretending to be Kati or the Bank, he can send fake messages\footnote{Moreover, the Adversary can deploy some malware on Kati's computer, but this dangerous threat has not much to do with mathematics and so it is disregarded in the paper.}.   As exchanging session keys is meaningful only if these keys are used, \eqref{eq:lrJgfncNpZnkt} combines this exchange with the use of section keys.
Observe that each of $\vec u$ and $\vec x$ can take $|\Boo n|^b=2^{n b}$ many values.  If $n b$ is small, then a random $nb$-dimensional vector equals   $\vec x$  with probability $2^{-n b}$, which cannot be neglected. So if $n b$ is small, then the Adversary can experiment with a random $\vec x$ and he succeeds in breaking the protocol too often.  In particular, we note for later reference that  
\begin{equation}
\parbox{8cm}{if $n=2$,  and $b=2$,  then, on average,  the Adversary can  break the protocol in every sixteenth step;} 
\label{eq:hlMrwzfg}
\end{equation}
this would allow him to  steal the money from Kati's bank account easily. It is reasonable to believe that if $nb$ is large, say, if $nb\geq 5000$, then the protocol outlined above is safe. As there are  $2^{nb}$  many candidate vectors for $\vec x$, the Adversary cannot try each of them in his lifetime (in fact, not even in the expected lifetime of the Solar System). Similarly, there are so many possible candidates for the master key $\vec h$ that the Adversary cannot find it by random trials.

Observe that the Adversary is not in the position to crack the  cipher $\Enc$ with a chosen plaintext attack.  Indeed, 
\begin{equation}
\parbox{9.8cm}{if he alters the $\vec p$ component in a message $(\vec p, \, \Enc(\vec u,\vec x))$, then the addressee works with a corrupted session key, decrypts the message into something nonsense, and stops communication. The same happens if the Adversary changes $(\vec p, \, \Enc(\vec u,\vec x))$  in any other way or he tries to send his own message; all he can learn is mentioned in  Footnote \ref{foot:clLtxr} but this  is not enough for him.
}
\label{eq:tljShnmDblM}
\end{equation} 
The chance that  an altered or Adversary-made message leads to a meaningful decrypted text can be neglected.  
However, the Adversary might have a chance to focus on $\Enc$ in case he conjectures  what $\vec x$ could be;  he can even try thousands of candidate $\vec x$'s. For example, a few days before the deadline that all  citizens have to pay 500 Euros as a local tax,  it is reasonable to guess that the clear text in Kati's message is ``transfer 500 Euros to the revenue account of the city''. So there is a real chance that 
\begin{equation}
\text{the Adversary knows a triple }\,\, \bigl(\vec p,\, \Enc(\vec p(\vec h),\vec x), \, \vec x)\bigr).
\label{eq:nRsTbRcpnbkBN}
\end{equation}
Observe that 
\begin{equation}
\parbox{10.3cm}{if $\Enc$ is appropriately chosen, which we assume, then the Adversary hardly has any chance to extract $\vec u:=\vec p(\vec h)$ from the last two components, $\Enc(\vec u,\vec x)$ and $\vec x$,  of the triple mentioned in \eqref{eq:nRsTbRcpnbkBN}.}
\label{eq:kvhttJndKrszrmgx}
\end{equation}
In other words, we chose an $\Enc$ that is not vulnerable for \emph{key-recovery attacks}.
In Section \ref{sect:NPh}, we are going to have a closer look at the situation when
\begin{equation}
\text{the Adversary acquires a pair $(\vec p, \vec u)=(\vec p, \vec p(\vec h))$,}
\label{eq:nXpvrLn}
\end{equation}
but let us emphasize here that it is very unlikely that the Adversary reaches \eqref{eq:nXpvrLn} in his lifetime. 
Even if the Adversary collects this sort of information several times, this is still insufficient  for him since  there are astronomically many possible master keys. Therefore, we believe that protocol \eqref{eq:lrJgfncNpZnkt} is very safe and, in particular, it is safer than protocol $\Enc$.

Vernam's cipher is known to be far  from withstanding a key-recovery attack in situation \eqref{eq:kvhttJndKrszrmgx}, that is, when the Adversary knows $\Enc(\vec u,\vec x)=\vec u+\vec x$ and $\vec x$. Indeed, then  he gets  $\vec u$ by subtraction. The following ad hoc remark might help.

\begin{remark}\label{rem:thRmdzsrmSh}
Assume that  every clear text in our model is a string of characters (with ASCII codes) 
0, 1,\dots, 127.  (A value other then $128-1=127$ would not make much difference.)
Call these characters \emph{eligible characters}. Also, let us call the characters (with ASCII codes) 128, 129, \dots, 255  \emph{false characters}. Denote the clear text  by $\vec y$. Insert at least as many false characters into $\vec y$ as the number of its (eligible) characters in an 
\emph{appropriate\footnote{We suggest that the false characters and the eligible characters follow similar distributions; this makes it more difficult to separate these two sets of characters with statistical methods.} random way} to obtain another vector, $\vec x$.  For example, if $\vec y=(65,78,32,88,58)$, then 
\begin{align*}
\vec x\text{ can be }&(\red{130}, \red{156}, 65, 78, \red{201}, 32, \red{203},88, \red{241}, 58, \red{151})\cr
\text{or } &( 65, \red{143}, \red{179},78, \red{184}, 32,88, \red{182},  \red{252},\red{137}, 58),\text{ etc..}
\end{align*}
Now if Kati wants to send the clear text $\vec y$ to the Bank, then she turns $\vec y$ to a random $\vec x$ described above and applies protocol  \eqref{eq:lrJgfncNpZnkt}. 
So Kati sends $\Enc(\vec u,\vec x)$ to the Bank. Then Bank, armed with $\vec p$ and $\vec u=\vec p(\vec h)$, decrypts $\Enc(\vec u,\vec x)$ to $\vec x$ as before  and obtains $\vec y$ by  omitting the false characters from $\vec x$.
\end{remark}

\subsection{A weaker version}\label{subsect:weakest}
In this version, the only purpose is \emph{authentication}; that is, Kati wants to prove her identity to the Bank. This protocol has nothing to do with session keys or $\Enc$. The parameters and the notations are the same an in \eqref{eq:lrJgfncNpZnkt}. 
\begin{equation}
\parbox{10.0cm}{To authenticate herself, Kati asks a random $\vec p$ from the Bank and
after receiving $\vec p$, she computes and sends  $\vec p(\vec h)$ to the Bank.}
\label{eq:weakvers}
\end{equation}
The idea is that  the Bank assumes that (except for the Bank itself) only Kati knows $\vec h$ and only she can compute $\vec p(\vec h)$. In our situation, Kati and the Bank do not have a symmetric role since the Bank does not authenticate itself. (When banking online, Kati avoids fake sites and clicks only on the real website of the Bank since her contract with the Bank,   signed when she opened the account, contains the Bank's URL.)
If the Adversary is passive, that is, if he can monitor messages but cannot alter them and cannot send his own messages, then it is only \eqref{eq:nXpvrLn} where he can get to.
Indeed, as it is clear from \eqref{eq:weakvers} and even from Footnote \ref{foot:bnkprvlgm},
the Adversary cannot choose $\vec p$ himself. Furthermore, the argument in \eqref{eq:tljShnmDblM} shows that the Adversary cannot alter a message in any other way. Hence, Kati has no reason to afraid of an active adversary.

However, if protocol \eqref{eq:weakvers} is adapted by two equal communicating parties so that each of them can play the role of the Bank, then there can be an active adversary, who  can alter and  send messages. For this Adversary, the chance to find the (master and only) key $\vec h$ is now better than in \eqref{eq:nXpvrLn} since he can choose $\vec p$. For this situation, we cannot prove anything but we guess that even the authentication protocol \eqref{eq:weakvers} would be safe in case of an appropriate strategy of choosing $\vec p$.

\begin{remark}
For authentication, it is safer to  send a message something like ``This is Kati'' encrypted using protocol \eqref{eq:lrJgfncNpZnkt}  fortified with Remark \ref{rem:thRmdzsrmSh}  rather than going after  \eqref{eq:weakvers}.
 This method works even if \eqref{eq:lrJgfncNpZnkt} is adapted by two equal communicating parties.
\end{remark}

\section{Even an easier problem is hard}\label{sect:NPh}
This section is motivated by the Adversary's problem in the situation described in 
\eqref{eq:nXpvrLn}; we study the hardest case of this problem. Note that the result of this section is meaningful, although less motivated, even without referencing Section \ref{sect:ACprtcl}. 

As in Section \ref{sect:ACprtcl}, we will assume that $n,k,b\in\Nplu$, $\vec p$ is a $b$-dimensional vector of $k$-ary lattice terms, and $\vec u\in \Boo n^b$. Writing $\vec x=(x_1,\dots,x_k)$  instead of $\vec h\in \Boo n^k$, the problem corresponding to \eqref{eq:nXpvrLn}  is this:
\begin{equation}
\CPr(n,b):\quad
\parbox{9cm}{given an input  $\vec p(\vec x)=\vec u$ with 
$\vec u\in \Boo n ^b$, find a solution of  the system  $\vec p(\vec x)=\vec u$ of equations for the unknown 
$\vec x\in \Boo n ^k$ 
\emph{in those cases} where there exists a solution.} 
\label{eq:sJbKmlDks}
\end{equation}
With the same meaning of $n$, $k$, $b$,  $\vec p$, and $\vec u$, we also define a related decision problem:
\begin{equation}
\DPr(n,b):\quad
\parbox{9cm}{given an input  $\vec p(\vec x)=\vec u$  with $\vec u\in \Boo n ^b$, decide whether the system $\vec p(\vec x)=\vec u$ of equations has a solution in $\Boo n ^k$ for the unknown $\vec x$.}
\label{eq:lnRbzkTjl}
\end{equation}
The acronyms $\CPr$ and $\DPr$ come from ``Construction Problem'' and ``Decision Problem'', respectively. Let $\size{\vec p(\vec x)=\vec u}$ and $\size{\vec h}$
denote the \emph{size} of $\vec p(\vec x)=\vec u$ and that of $\vec h$, respectively; these sizes are the respective numbers of bits; see \eqref{eq:mrtmndksklD}.

There are many books and papers dealing with the widely known concept of the complexity classes \clp{} and \clnp{}; some of them will be cited later but even \href{https://en.wikipedia.org/wiki/NP-completeness}{Wikipedia} is sufficient for us.  However, \clp{}, \clnp{}, and \clnp{}-completeness are usually about \emph{decision problems} 
while $\CPr(n,b)$  in \eqref{eq:sJbKmlDks} is not such.  There is another difference: while we require an answer for \emph{each input string} in case of a decision problem, this is not so in case of $\CPr(n,b)$. (In particular,  an algorithm solving $\CPr(n,b)$ need not even halt if $\vec p(\vec x)=\vec u$ unsolvable.)  These circumstances constitute our excuse that we neither define what the \clnp{}-completeness of  $\CPr(n,b)$ could mean nor we know whether $\CPr(n,b)$ would have such a property (as we would experience difficulty with a suitable replacement of $\Alg_1(d)$ later in the proof). However, we can safely agree to the following terminology:
\begin{equation}
\parbox{11.2cm}
{$\CPr(n,b)$, given in \eqref{eq:sJbKmlDks}, \emph{is solvable in polynomial time}  \quad $\defiff$ \quad there are an algorithm $\Alg(n,b)$ and a polynomial $f^{(n,b)}$ such that for every  input equation $\vec p(\vec x)=\vec u$ of $\CPr(n,b)$, \emph{if}  $\vec p(\vec x)=\vec u$ has a solution, 
  \emph{then}  $\Alg(n,b)$ finds one of its solutions in (at most) $f^{(n,b)}(\size{\vec p(\vec x)=\vec u})$ steps.}
\label{eq:jrDljrkthlH}
\end{equation}
The algorithm and the polynomial depend on the parameters $n$ and $b$. We could have written  ``time'' instead of ``steps''. Later, we will always omit ``(at most)''.  

We have the following statement, in which $b$ denotes the dimension of $\vec p$.

\begin{proposition}\label{prop:NP} For $2\leq b\in\Nplu$ and $n\in\Nplu$, if $\CPr(n,b)$, defined in  \eqref{eq:sJbKmlDks}, is solvable in polynomial time  then \clp{} is equal to \clnp{}.
\end{proposition}

Even if the famous ``is \clp{} equal to \clnp{}?'' problem is, unexpectedly, solved affirmatively in the future, the \emph{proof} below will still say something on the difficulty of  $\CPr(n,b)$.

\begin{proof}
In the whole proof, we assume that $2\leq b\in\Nplu$, $n\in\Nplu$, and  $\CPr(n,b)$ is solvable in polynomial time. 

In principle, we should have written ``Turing machine'' in  \eqref{eq:jrDljrkthlH} rather than ``algorithm''\footnote{and ``input string'' rather than ``input equation'', but this distinction would not make an essential difference as the syntax of the input string can be checked in polynomial time.}. Fortunately, the algorithms in the proof (which are clearly equivalent to usual computer programs) can be simulated by Turing machines and this simulation preserves the property ``in polynomial time''; see, for example, Theorem 17.4 in Rich \cite{rich}. By the same theorem, for $n'$ computer steps%
\footnote{We can think of the commands in low-level computer programming languages but not of  compound commands like ``NextPrimeAbove(n)'' of ``InvertMatrix(A)'' in high-level programming languages.}  (and for $n'$ steps in our mind), the simulating Turing machine  needs $(O(n'))^6$ steps.  Therefore,  we will mostly  speak of polynomials without specifying their degrees even when a sub-algorithm is clearly linear (or even better) in our mind, that is, for our computers. For example, 
\begin{equation}
\parbox{11.4 cm}{for each fixed $d\in\Nplu$, there are  a polynomial $f^{(d)}_1$ and an algorithm $\Alg_1(d)$ such that, for  each $\xi\in\Nplu$, 
$\Alg_1(d)$ computes and stores $\xi^d$ in   $f_1^{(d)}(\xi)$ steps.}
\label{eq:sTdOssteps}
\end{equation}
Clearly, there are polynomials $f_2^{(n,b)}$ and $f_3$ and  algorithms $\Alg_2(n,b)$ and $\Alg_3$ such that for all  inputs  $\vec p(\vec x)=\vec u$, as in \eqref{eq:sJbKmlDks}, and   $\vec h\in \Boo n ^k$,
\begin{align}
&\parbox{8cm}{$\Alg_2(n,b)$ decides in  $f_2^{(n,b)}\bigl(\size{\vec p(\vec x)=\vec u}+\size{\vec h}\bigr)$ steps whether $\vec h$ is a solution of $\vec p(\vec x)=\vec u$, and}
\label{eq:xshklXpslg}\\
&\parbox{8cm}{$\Alg_3$ computes and stores the number $\size{\vec p(\vec x)=\vec u}$  in  $f_3\bigl(\size{\vec p(\vec x)=\vec u}\bigr)$ steps.}
\label{eq:dpRkCdVr}
\end{align}
Let $\Alg(n,b)$ and $f^{(n,b)}$ be chosen according to \eqref{eq:jrDljrkthlH}. We can assume that $f^{(n,b)}$ is of the form $f^{(n,b)}(\xi)=\xi^{d(n,b)}$ for some $d(n,b)\in\Nplu$.  Then  $\Alg(n,b)$ halts in $\bigl(\size{\vec p(\vec x)=\vec u}\bigr){}^{d(n,b)}$ steps for any solvable input $\vec p(\vec x)=\vec u$ but we do not know what  $\Alg(n,b)$ does and whether it ever  halts at other inputs. Using \eqref{eq:sTdOssteps}--\eqref{eq:dpRkCdVr}, we define another algorithm $\Blg(n,b)$ as follows. 
{The input of $\Blg(n,b)$ is a system of equations $\vec p(\vec x)=\vec u$ from}  \eqref{eq:lnRbzkTjl}; let $s:= \size{\vec p(\vec x)=\vec u}$.
The first task of $\Blg(n,b)$ is to save a copy of  $\vec p(\vec x)=\vec u$; this needs $ f_0(s)$ steps where $f_0$ is a polynomial not depending on the parameters $n$ and $b$ and the input  $\vec p(\vec x)=\vec u$. The second part of $\Blg(n,b)$ is $\Alg_3$, which borrows the  
 input $\vec p(\vec x)=\vec u$ from $\Blg(n,b)$ and puts  $s$ to the output stream  in $f_3(s)$ steps. 
The next part of $\Blg(n,b)$ is $\Alg_1(d(n,b))$, which considers the output of $\Alg_3$ as an input and puts $f^{(n,b)}(s)=s^{d(n,b)}$ into a (counter) variable $c$ in $f_1^{(d(n,b))}(s)$ steps. Then $\Blg(n,b)$ performs the steps of $\Alg(n,b)$ and the ``$(\alpha)$--$(\delta)$-strides'' given below alternately. (Here a  ``stride'' means a finite sequence of steps, possibly just one step.)  After the first $\Alg(n,b)$-step, $\Blg(n,b)$ performs the following strides.
\begin{enumerate}
\item[($\alpha$)] $\Blg(n,b)$ decreases $c$ by 1.
\item[($\beta$)] $\Blg(n,b)$ verifies whether $c = 0$.
\item[($\gamma$)] $\Blg(n,b)$ checks whether $\Alg(n,b)$ has halted.
\item[($\delta$)] If $c=0$ or  $\Alg(n,b)$ has halted then, using the saved copy of  $\vec p(\vec x)=\vec u$,  $\Blg(n,b)$ executes $\Alg_2(n,b)$ to verify whether the output of $\Alg$ is a solution of  $\vec p(\vec x)=\vec u$. If $\Alg_2(n,b)$ terminates with ``yes'', then $\Blg(n,b)$ outputs ``yes, the equation is solvable'' and halts. Otherwise, if $\Alg_2(n,b)$ terminates with ``no'', then $\Blg(n,b)$ outputs ``no, the equation is not solvable'' and halts.
\end{enumerate}
After these $(\alpha)$--$(\delta)$-strides, the next $\Alg(n,b)$-step is performed, then the $(\alpha)$--$(\delta)$-strides again, etc. 
The \emph{kernel} of the  $(\delta)$-stride is its part following the \emph{premise}  ``if $c=0$ or  $\Alg(n,b)$ has halted''; this kernel is performed only once.  
As $c\leq s^{d(n,b)}$, there is a polynomial $f_4^{(n,b)}$, not depending on the input of $\Blg(n,b)$, such that each of the $(\alpha)$--$(\gamma)$-strides can be done in $f_4^{(n,b)}(s)$ many  steps and, furthermore, the same holds for every $\Alg$-step (since it is only a one-step stride) and for  the condition  part of $(\delta)$. The $\Alg$-step, $(\alpha)$, $(\beta)$, $(\gamma)$, and the premise of $(\delta)$ are performed $f^{(n,b)}(s)=s^{d(n,b)}$ times, each.  
The kernel of the $(\delta)$-part, which is performed only once, is the same as $\Alg_2(n,b)$. The input of $\Alg_2(n,b)$ in this case is (the saved copy of) $\vec p(\vec x)=\vec u$ (of size $s$) together with $\vec h$, taken from the output stream of $\Alg(n,b)$. (Even if $\Alg(n,b)$ does not halt, there is a memory space --- or, in case of a Turing machine,  there is an output tape --- where $\vec h$ is expected when it exists.) As an element of $\Boo n$ can be stored in $n$ bits, $\size{\vec h}=n k$. Here $n$ is a constant and $k\leq s$ since $\vec x$ has $k$ components that occur in $\vec p=(\vec x)=\vec u$. Hence,
$\size{\vec h}\leq n s$, whereby $\Alg_2(n,b)$ decides in 
$f_2^{(n,b)}(s+n s)=f_2^{(n,b)}((n+1)s)$ steps whether the output of $\Alg(n,b)$ is a solution of our system of equations.  Therefore, $\Blg(n,b)$ 
 halts after
\begin{equation}
g^{(n,b)}(s):=f_0(s)+f_3(s)+f_1^{(d(n,b))}(s)+ f^{(n,b)}(s)\cdot f^{(n,b)}_4(s)+f_2^{(n,b)}((n+1)s)
\label{eq:lkMdszKznm}
\end{equation}
steps. As we treat the parameters $n$ and $b$ as constants,  $g^{(n,b)}$ is a univariate polynomial. Since the simulated $\Alg$ finds any solution before the counter $c$ becomes 0, $\Blg$ correctly decides whether $\vec p(\vec x)=\vec u$ has a solution or not. That is, $\Blg$ solves $\DPr(n,b)$. We have seen that  $g^{(n,b)}$  in \eqref{eq:lkMdszKznm} is a polynomial, whereby
\begin{equation}
\parbox{7cm}{$\DPr(n,b)$, defined in \eqref{eq:lnRbzkTjl}, is in \clp{}, and $\Blg(n,b)$ solves it in $g^{(n,b)}($input size) steps.}
\label{eq:szKnmRrlSTgF}
\end{equation}

As the next step of the proof, we focus on another problem. 
An input of the \emph{$3$-coloring problem} 
is a finite  graph $G=(\set{1, \dots, t},E_G)$, where 
$t\in\Nplu$ and the edge set $E_G$ consists of some two-element subsets of $\set{1,\dots,t}$. 
By a \emph{$3$-coloring}  we mean a sequence $C_1$, $C_2$,  \dots, $C_t$ of nonempty subsets of $\set{r,w,g}:=\{$red, white, green$\}$ such that whenever $\set{i,j}\in E_G$, then $C_i\cap C_j=\emptyset$. (This is 
equivalent to the original definition, where each vertex has exactly one color
since we can change a color $\xi$ to $\set{\xi}$ and, in the converse direction, we can take the lexicographically first element of each nonempty subset of $\set{r,w,g}$.)

To reduce the 3-coloring problem to problem $\DPr(n,b)$, let $G$ be the graph from the previous paragraph, and let $\sg :=\size G$.
Let  $r_1$, $w_1$, $g_1$,  \dots, $r_t$, $w_t$, $g_t$ be variables;  their task is to determine a 3-coloring. These $k:=3t$  variables form the components of a vector denoted by $\vec x$.  For each vertex $v\in\set{1,\dots,t}$ and each edge $\set{i,j}\in E_G$, consider the $k$-ary lattice terms 
\begin{equation}
  a_{v}(\vec x):= r_v\vee  w_v\vee  g_v \text{ and }
  b_{i j}(\vec x):=(r_i\wedge r_j)\vee (w_i\wedge w_j)\vee
(g_i\wedge g_j).
\label{eq:mklSkdnmRbg}
\end{equation}
For $m\in\set{2,\dots,t}$, let
\begin{equation}
p_1:=\bigwedge\set{a_{v}(\vec x):v\in\set{1,\dots,t}}\quad
\text{ and }\quad
p_m:=\bigvee\set{b_{i j}(\vec x):\set{i,j}\in E_G},
\label{eq:dmNrmHnJs}
\end{equation}
$\vec p:=(p_1,\dots,p_t)$, and $\vec u=(u_1,\dots,u_t):=(1,0,\dots,0)$, where\footnote{Note that, to reduce the size of $\vec p$, we could have let $p_3=\dots=p_t:=r_1\vee w_1\vee g_1$ together with $u_3=\dots=u_t=1$.} $0=0_{\Boo n}$ and $1=1_{\Boo n}$. 
 We claim that 
\begin{equation}
\vec p(\vec x)=\vec u\text{ has a solution in }\Boo n^k\text{ if and only if } G\text{ is 3-colorable.}
\label{eq:hdnlVln}
\end{equation}
To see this, assume that $C_1$, \dots, $C_t$ are color sets witnessing that $G$ is 3-colorable. For $v\in\set{1,\dots, t}$, 
let  $r_v:=1\iff r\in C_v$,  $w_v:=1\iff w\in C_v$, and $g_v:=1\iff g\in C_v$.  If a variable is not 1, then let it be 0. Clearly, these assignments yield a solution in $\Boo n^k$ of $\vec p(\vec x)=\vec u$. 
Conversely, assume that $\vec p(\vec x)=\vec u$ has a solution $\vec x'=(r'_1,w'_1,g'_1,\dots, r'_t, w'_t, g'_t)\in\Boo n^k$ for the unknown $\vec x$, and fix an atom $e$ in $\Boo n$.
For each $v\in\set{1,\dots, t}$, 
define $C_v\subseteq \set{r,w,g}$ by the rules $r\in C_v\iff e\leq r'_v$, $w\in C_v\iff e\leq w'_v$, and  $g\in C_v\iff e\leq g'_v$. 
For any $v\in\set{1,\dots,t}$, $p_1(\vec x')=u_1=1$ and \eqref{eq:dmNrmHnJs} give that 
$e\leq 1=p_1(\vec x')\leq  a_v(\vec x')=r'_v\vee w'_v\vee g'_v$. Using the well-known fact that every atoms (and, in fact, any join-irreducible element) in a finite distributive lattice is join-prime, we obtain that at least one of the inequalities 
$e\leq r'_v$, $e\leq w'_v$, and $e\leq g'_v$ holds, whereby $C_v$ is nonempty. For $\set{i,j}\in E_G$, 
$p_2(\vec x')=u_2=0$ and \eqref{eq:dmNrmHnJs} give that $(r'_i\wedge r'_j)\vee (w'_i\wedge w'_j)\vee
(g'_i\wedge g'_j)=0$. 
Hence, $r'_i\wedge r'_j=w'_i\wedge w'_j=
g'_i\wedge g'_j=0$. If, say, we had that $r\in C_i\cap C_j$, then $e\leq r'_i$ and $e\leq r'_j$ would lead to
$e\leq r_i'\wedge r'_j=0$, a contradiction. Hence, $r\notin C_i\cap C_j$, and similarly for the colors $w$ and $g$, showing that $C_i\cap C_j=\emptyset$. So $C_1,\dots, C_t$  witness that $G$ is 3-colorable, and we have shown \eqref{eq:hdnlVln}.

Let  $s_G:=\size G$ and  $s$ stand for the size of $G$  and, complying with the earlier notation,  the size of  the equation in \eqref{eq:hdnlVln}, respectively. 
It is not hard to see that there are  polynomials $\mu$ and $\eta$ not depending on $G$ such that   $\vec p(\vec x)=\vec u$   can be constructed from $G$ in $\eta(\sg)$ steps and  $s\leq \mu(\sg)$. We define an algorithm $\Mlg$ as follows. 
For a graph  $G$ as an input, $\Mlg$ constructs  $\vec p(\vec x)=\vec u$, 
then it calls $\Blg(n,b)$ and, finally, it  outputs the same answer that $\Blg(n,b)$ has given. By \eqref{eq:szKnmRrlSTgF} and  \eqref{eq:hdnlVln},  $\Mlg$ solves the 3-coloring problem.  As  $s=\size{\vec p(\vec x)=\vec u}\leq \mu(\sg)$, $\Mlg$ does so in $\nu(\sg):= \eta(\sg) + g^{(n,b)}(\mu(\sg))$ steps. 
As $\nu$ is a polynomial, we obtain that the 3-coloring problem is in $\clp$. On the other hand, we know from Garey, Johnson, and Stockmeyer \cite{Gareyatal}, see also Dailey \cite[Theorems 3 and 4]{dailey}, that $3$-coloring is an \clnp{}-complete problem. Now that an  \clnp{}-complete problem turned out to be in \clp, it follows that \clnp{} = \clp{}, completing the proof. 
\end{proof}

\begin{remark}\label{rem:szrltJLz}
The proof above has reduced the \clnp-complete 3-coloring problem to problem $\DPr(n,b)$, defined in \eqref{eq:lnRbzkTjl}. Therefore,  $\DPr(n,b)$  is also an \clnp-complete  problem for any $2\leq b\in\Nplu$ and any $n\in\Nplu$. 
\end{remark}

Even more is true since the only property of $\Boo n$ that the proof used is that $e$ is join-prime, which means that $\set{x: x\geq e}$ is a prime filter. 

\begin{remark}\label{rem:szrcskHlr}
In every finite  lattice $L$ that has a prime filter and for each $2\leq b\in\Nplu$, the solvability of systems of $b$ equations with constant-free left sides but constant right sides is an \clnp-complete problem. 
\end{remark}

\section{Warning and perspectives}\label{section:warn}
Sometimes, cryptography  goes after conjectures and experience if no  rigorous mathematical proof is available. For example, we only \emph{believe} that the RSA crypto-system is safe and \clp{} $\neq$ \clnp{}. This can justify that no proof occurs in Sections \ref{sect:ACprtcl} and \ref{section:warn}. However, we have some comments.

\begin{remark}\label{rem:sltRsn} 
Even though it is harder to break protocol \eqref{eq:lrJgfncNpZnkt} than  to solve \eqref{eq:sJbKmlDks}
and we know from  Proposition \ref{prop:NP} that 
\eqref{eq:sJbKmlDks} is a hard problem,  these facts themselves do not guarantee the safety of protocol \eqref{eq:lrJgfncNpZnkt}.
\end{remark}
This is so because of (at least) two reasons. First, the Adversary might break  protocol \eqref{eq:lrJgfncNpZnkt} without solving \eqref{eq:sJbKmlDks}. For example, we know from Proposition \ref{prop:NP} that  \eqref{eq:sJbKmlDks} is hard even with the parameters given in  \eqref{eq:hlMrwzfg} but  \eqref{eq:lrJgfncNpZnkt} can easily be broken in this case. 
Second, a  hard problem and even an \clnp{}-complete problem can have many easy instances (that is, inputs) for which the computation is fast and even an ``average''  instance can be such; see the Introduction in Wang \cite{wang} for details. 

Remark \ref{rem:sltRsn} points out that we have not proved that protocols \eqref{eq:lrJgfncNpZnkt} and \eqref{eq:weakvers} are as safe as we believe. Furthermore, an exact proof would need a well-defined strategy of choosing $\vec p$, and a definition to capture what safety means.   Actually, such a definition exists, see Levin \cite{levin} and see also Wang \cite{wang}, but we do not need its complicated details here. Some ideas about a strategy (in another lattice) are outlined in 
Cz\'edli \cite{czgDEBRauth} but without any proof, and we do not know whether  these ideas can be supported by a proof. This is why we mention the \emph{tiling problem} from  Levin \cite{levin}; see also Wang \cite{wang} as a secondary source. This problem, which we do not define here, includes a probabilistic distribution. Levin \cite{levin} proved that, with respect to this distribution, the average case of the tiling problem is hard in some (sophisticated) sense. 
Similarly to the proof of Proposition \ref{prop:NP}, see also Remark \ref{rem:szrltJLz},  
we could reduce the tiling problem to $\DPr(n,b)$ defined in   \eqref{eq:lnRbzkTjl}  but this would require too much tedious work for a lattice theoretic paper. 
(The \clnp-completeness of $\DPr(n,b)$ implies the existence of such a reduction but we need a concrete one that is sufficiently economic.)
Then we could pick a random $\vec p$ for  \eqref{eq:lrJgfncNpZnkt} so that first we take a random instance $y$ of the tiling  problem and then we let $\vec p$ be the polynomial vector in the ``$\DPr(n,b)$-representative'' of $y$. As $y$ and, thus, the corresponding equation in $\DPr(n,b)$ are hard on average, we can hope
that this $\vec p$ turns $\CPr(n,b)$, a problem easier than \eqref{eq:lrJgfncNpZnkt}, hard. 
However, the details of this plan have not been elaborated yet. In particular, we have not proved that the above-suggested method of choosing $\vec p$ (which is only a part of the $\DPr(n,b)$-representative of $y$)  turns $\CPr(n,b)$ (which is another problem)  hard on average. Furthermore,  
it is not clear whether the parameters suggested in Section \ref{sect:ACprtcl} are large enough for the plan suggested above; enlarging the paramaters reduce the practical value of \eqref{eq:lrJgfncNpZnkt} .

Finally, we note that protocols \eqref{eq:lrJgfncNpZnkt}  and \eqref{eq:weakvers} become more economic if we decrease $k$ so that $\Boo n$ still has very many $k$-dimensional generating vectors; this is the point where Sections \ref{sec:genBoole} and \ref{sect:prrslt} are connected to the  Section \ref{sect:ACprtcl}.

\end{document}